\documentclass[final,3p,times]{article}

\usepackage{lineno,hyperref}
\modulolinenumbers[1]

\usepackage{graphicx}
\usepackage{amssymb}
\usepackage{mathtools}
\usepackage{amsthm}
 \usepackage{amsmath}
\usepackage{dsfont}
\usepackage{epsfig}
\usepackage{float}
\usepackage{epstopdf}
\usepackage{subfigure}
\usepackage{cite}
\usepackage{xcolor}
\usepackage{latexsym,amssymb}
\usepackage{pgfplots}
\usepackage{amsmath}
\usepackage{authblk}

\newcommand{\gauss}[2]{\genfrac{[}{]}{0pt}{}{#1}{#2}_q}
\newcommand{\vp}{\varphi} 
\newcommand{\al}{\alpha}
 
\newcommand{\la}{\lambda}

\numberwithin{equation}{section}

\theoremstyle{definition}

\newtheorem{theorem}{Theorem}[section]
\newtheorem{lemma}[theorem]{Lemma}
\newtheorem{proposition}[theorem]{Proposition}
\newtheorem{corollary}[theorem]{Corollary}
\newtheorem{definition}[theorem]{Definition}
\newtheorem{remark}[theorem]{Remark}
\newtheorem{exmp}[theorem]{Example}

\theoremstyle{remark}

	\usetikzlibrary{plotmarks}
\usetikzlibrary{external}
\pgfplotsset{compat=1.3}
\newlength\figurewidth 
\newlength\figureheight
\tikzset{external/force remake=true}
\usepackage{microtype} 

\usepackage[margin=0.88in]{geometry}

\date{}

\begin{document}



\title{$h$--$\gamma$ Blossoming, $h$--$\gamma$ Bernstein Bases, and $h$--$\gamma$ B\'{e}zier Curves for Translation Invariant $\left(\gamma_{1},\gamma_{2}\right)$ Spaces}

\author[a]{Fatma Z\"{u}rnac{\i}-Yeti\c{s}}
\author[b]{Ron Goldman}
\author[c]{Plamen Simeonov \footnote{\textbf{Email addresses:} $^a$fzurnaci@itu.edu.tr, $^b$rng@cs.rice.edu  $^c$simeonovp@uhd.edu, }}

\affil[a]{Department of Mathematics Engineering, Istanbul Technical University,  Maslak, Istanbul, 34469, Turkiye}

\affil[b]{Department of Computer Science, Rice University, Houston, Texas, 77251, USA}

\affil[c]{Department of Mathematics and Statistics, University of Houston-Downtown, Houston, Texas 77002, USA}

\maketitle

\begin{abstract}
	A $\left(\gamma_{1}, \gamma_{2}\right)$ space of order $n$ is a space of univariate functions spanned by $\left\{\gamma_{1}^{n-k}(x), \gamma_{2}^{k}(x)\right\}_{k=0}^{n}$. A $\left(\gamma_{1}, \gamma_{2}\right)$ space is said to be translation invariant if $\gamma_{1}(x-h)$ and $\gamma_{2}(x-h)$ can be expressed as nonsingular linear combinations of $\gamma_{1}(x)$ and $\gamma_{2}(x)$. Translation invariant $\left(\gamma_{1}, \gamma_{2}\right)$  spaces include polynomials $\left(\gamma_{1}(x)=1, \gamma_{2}(x)=x\right)$, trigonometric functions $\left(\gamma_{1}(x)=\cos x, \gamma_{2}(x)=\sin x\right)$, hyperbolic functions $\left(\gamma_{1}(x)=\cosh x, \gamma_{2}(x)=\sinh x\right)$, and their discrete analogues. We merge $\gamma$-blossoming for $\left(\gamma_{1}, \gamma_{2}\right)$ spaces with $h$-blossoming for $h$-Bernstein bases and $h$-B\'{e}zier curves to construct a novel $h$--$\gamma$ blossom for translation invariant $\left(\gamma_{1}, \gamma_{2}\right)$ spaces generated by two continuous, linearly independent
	functions $\gamma_{1}$ and $\gamma_{2}$. Based on this $h$--$\gamma$ blossom, we define $h$--$\gamma$ Bernstein bases and $h$--$\gamma$ B\'{e}zier curves and study their properties. We derive recursive evaluation algorithms, subdivision procedures, Marsden identities, and formulas for degree elevation and interpolation for these $h$--$\gamma$ Bernstein and $h$--$\gamma$ B\'{e}zier schemes.
\end{abstract}
 




\section{Introduction}\label{section1}
A $\left(\gamma_{1}, \gamma_{2}\right)$ space of order $n$ is a space of univariate functions given by $\pi_{n}\left(\gamma_{1}, \gamma_{2}\right)=\text{span}\left\{\gamma_{1}^{n-k}(x), \gamma_{2}^{k}(x)\right\}_{k=0}^{n}$. A $\left(\gamma_{1}, \gamma_{2}\right)$ space is said to be translation invariant if $\begin{pmatrix}\gamma_1(x-h)\\\gamma_2(x-h)\end{pmatrix}
= C(h) \begin{pmatrix}\gamma_1(x)\\\gamma_2(x)\end{pmatrix}$, where $C(h)$ is a nonsingular $2 \times 2$ matrix. Many common spaces are translation invariant $\left(\gamma_{1}, \gamma_{2}\right)$ spaces, including polynomials  (\(\gamma_{1}(x)=1\), \(\gamma_{2}(x)=x\)),
trigonometric functions (\(\gamma_{1}(x)=\cos x\), \(\gamma_{2}(x)=\sin x\)),
hyperbolic functions (\(\gamma_{1}(x)=\cosh x\), \(\gamma_{2}(x)=\sinh x\)), and spaces of the discrete analogues of trigonometric and hyperbolic functions (see Section \ref{section2}).

The $\left(\gamma_{1}, \gamma_{2}\right)$ spaces were first introduced by Gonsor and Neamtu in \cite{gonsor} to extend polynomial techniques such as blossoming to more general spaces such as spaces of trigonometric functions. Lyche \cite{lyche} also formulates a notion of blossoming for trigonometric functions in order to study properties of trigonometric splines. An alternative approach to blossoming for $\left(\gamma_{1}, \gamma_{2}\right)$ spaces is provided by Di{\c{s}}ib{\"u}y{\"u}k  and Goldman in \cite{cetin2,cetin3}.

The $h$-blossom was first introduced by Simeonov et al in \cite{ron1} to study the $h$-Bernstein bases and $h$-B\'{e}zier curves initiated in Approximation Theory by  Stancu \cite{stancu1,stancu2}, and rediscovered in CAGD by Goldman \cite{goldman1}, and Goldman and Barry \cite{goldman2}. The $h$-Bernstein bases and $h$-B\'{e}zier curves, however, are still polynomial schemes albeit represented in a different basis from classical Bernstein-B\'{e}zier schemes.

\textit{The goal of this paper is to merge $\gamma$-blossoming for $\left(\gamma_{1}, \gamma_{2}\right)$ spaces with $h$-blossoming for $h$-Bernstein bases and $h$-B\'{e}zier curves to construct a novel $h$--$\gamma$ blossom and novel $h$--$\gamma$ Bernstein bases and $h$--$\gamma$ B\'{e}zier curves for  translation invariant $\left(\gamma_{1}, \gamma_{2}\right)$ spaces generated by two continuous, linearly independent
		functions $\gamma_{1} $ and $ \gamma_{2}$}.  The parameter $h$ serves not only as a new shape parameter for $\left(\gamma_{1}, \gamma_{2}\right)$ spaces, but also more importantly allows for interpolation formulas not present in conventional $\gamma$-Bernstein and $\gamma$-B\'{e}zier schemes.
 
 We begin in Section \ref{section2} by introducing the notion of translation invariance and providing several examples of translation invariant $\left(\gamma_{1}, \gamma_{2}\right)$ spaces, including classical trigonometric and hyperbolic spaces and their discrete analogues. In Section \ref{section3} we define the $h$--$\gamma$ blossom and establish the existence and uniqueness of this $h$--$\gamma$ blossom for translation invariant $\left(\gamma_{1}, \gamma_{2}\right)$ spaces generated by two continuous, linearly independent functions $\gamma_{1} $ and $ \gamma_{2}$. (In an Appendix we provide conditions which help to guarantee the existence of this blossom by establishing conditions that guarantee the linear independence of certain functions used in the proof of existence.). We also provide explicit formulas for the $h$--$\gamma$ blossom of certain key functions. Section \ref{section4} is devoted to developing recursive formulas for evaluating the $h$--$\gamma$ blossom. 

In Section \ref{section5}, we introduce the $h$--$\gamma$ Bernstein basis functions and $h$--$\gamma$ B\'{e}zier curves. We establish a two-term recurrence for the $h$--$\gamma$ Bernstein basis functions analogous to the standard two-term recurrence for the classical Bernstein basis functions. We show that the blossom evaluated at special values provides the dual functionals (i.e. coefficients in the $h$--$\gamma$ Bernstein basis) for $h$--$\gamma$ B\'{e}zier curves. We then use this dual functional property to derive recursive evaluation algorithms, subdivision procedures, Marsden identities, and
formulas for degree elevation and interpolation for these $h$--$\gamma$ Bernstein and $h$--$\gamma$ B\'{e}zier schemes.

\section{Translation Invariance}\label{section2}
In all of the results in this paper we will assume that the functions $\gamma_1$ and $\gamma_2$ are defined and continuous on an open set which contains all the parameter values where these functions are evaluated.
\begin{definition}
	The functions	$\Gamma=(\gamma_{1}(x),\gamma_{2}(x))$ are translation invariant if  there exists an invertible $2\times 2$ matrix function $C=C(h)$ such that
	\[
	\begin{pmatrix}\gamma_1(x-h)\\[2pt]\gamma_2(x-h)\end{pmatrix}
	= C \begin{pmatrix}\gamma_1(x)\\[2pt]\gamma_2(x)\end{pmatrix}.
	\]
\end{definition} 
We will use translation invariance to establish the existence of the $h$--$\gamma$ blossom.
\begin{exmp}[\textbf{Polynomials}]\label{ex1}
	The functions $\Gamma= (1, x)$ are translation invariant.
\end{exmp}
\begin{exmp}[\textbf{Classical Trigonometric Functions}]\label{e2}
	 The functions $\Gamma = (\cos(x) , \sin(x))$ are translation invariant, since
	\begin{align}\label{trig1}
		&\cos(x-h) = \cos h \hspace{0.05em}\cos x+\sin h \hspace{0.05em}\sin x, \\ \label{trig2}
		&\sin(x-h) = \cos h \hspace{0.05em} \sin x-\sin h \hspace{0.05em}\cos x.
	\end{align}
\end{exmp}
\begin{exmp}[\textbf{Discrete Trigonometric Functions}]\label{trigex}
	To construct discrete trigonometric functions, we begin with a discrete analogue of the exponential function. For $d>-1$, $d \neq 0$ define
	\begin{align*}
		 e_{d}=(1+d)^{1 / d},\,\,\, e_{d}^{x}=(1+d)^{x/d}.
	\end{align*}
	Notice that $e_{d} \rightarrow e$ and $e_{d}^{x} \rightarrow e^{x}$ as $d \rightarrow 0$.\\
	Now in analogy with the classical trigonometric functions, define the discrete versions of sine and cosine by setting
	\begin{align*}
		 \cos _{d}x=\frac{e_{d}^{i x}+e_{d}^{-i x}}{2}, \,\,\, \sin _{d}x=\frac{e_{d}^{i x}-e_{d}^{-i x}}{2 i}.
	\end{align*}
	Then it is straightforward to verify that these discrete versions of sine and cosine satisfy many of the same identities as the classical sine and cosine. In particular, (\ref{trig1}) and (\ref{trig2}) are satisfied with $\cos$ and $\sin$ replaced by $\cos_{d}$ and $\sin_{d}$.
	Therefore the functions $\Gamma=\left(\cos _{d}x, \sin _{d}x\right)$ are translation invariant.
\end{exmp}
\begin{exmp}[\textbf{Hyperbolic and Discrete Hyperbolic Functions}]\label{ex4}
The functions $\Gamma=(\cosh x, \sinh x)$ are translation invariant, since
\begin{align}\label{hyperid1}
	\cosh (x-h)=\cosh h \hspace{0.05em}\cosh x-\sinh h \hspace{0.05em} \sinh x, & \\\label{hyperid2}
	\sinh (x-h)=\cosh h \hspace{0.05em}\sinh x-\sinh h \hspace{0.05em} \cosh x. &
\end{align}
Let $e_{d}^{x}$ be as in Example \ref{trigex}. The discrete versions of $\cosh$ and $\sinh$ are defined in analogy with the classical versions of $\cosh$ and $\sinh$ by setting
$$
\cosh _{d}x=\cos_{d}(ix)=\frac{e_{d}^{x}+e_{d}^{-x}}{2} \quad \text { and } \quad \sinh _{d}x=i \sin_{d}(ix)=\frac{e_{d}^{x}-e_{d}^{-x}}{2} .
$$
Now it is once again straightforward to verify from these definitions that these discrete versions of $\cosh$ and $\sinh$ satisfy many of the same identities as the classical versions of $\cosh$ and $\sinh$. In particular,  (\ref{hyperid1}) and (\ref{hyperid2}) are satisfied with $\cosh$ and $\sinh$ replaced by $\cosh_{d}$ and $\sinh_{d}$. Therefore the functions $\Gamma=\left(\cosh _{d}x, \sinh _{d}x\right)$ are translation invariant.
\end{exmp}
\begin{exmp}[\textbf{Products of Exponentials with Translation Invariant Functions}]
If the functions $\Gamma=\left(\gamma_{1}(x), \gamma_{2}(x)\right)$ are translation invariant, then the functions $\Gamma=\left(e^{x} \gamma_{1}(x), e^{x} \gamma_{2}(x)\right)$ and $\Gamma=\left(e_{d}^{x} \gamma_{1}(x), e_{d}^{x} \gamma_{2}(x)\right)$ are also translation invariant.
\end{exmp}
\section{Blossoming}\label{section3}
Di{\c{s}}ib{\"u}y{\"u}k and Goldman  \cite{cetin2} extended the notion of blossoming from the space of homogeneous
polynomials to the space $\pi_{n}(\gamma_{1},\gamma_{2})$ of homogeneous polynomials in the functions $(\gamma_{1},\gamma_{2})$. We
shall now construct an $h$-version  by altering the diagonal property of the non-polynomial
homogeneous blossom.

\begin{definition}
	Let $G(t)\in \pi_{n}(\gamma_{1},\gamma_{2})$. Then $g\left(\left(u_{1}, v_{1}\right), \ldots,\left(u_{n}, v_{n}\right); h\right)$ is an $h$--$\gamma$ blossom of $G$ if $g$ satisfies the following three axioms:
	
	\begin{enumerate}
		\item Symmetry: For every permutation $\sigma$ of $\{1, \ldots, n\}$
	\end{enumerate}\begin{align*}
		g\left(\left(u_{\sigma(1)}, v_{\sigma(1)}\right), \ldots,\left(u_{\sigma(n)}, v_{\sigma(n)}\right) ; h\right)=g\left(\left(u_{1}, v_{1}\right), \ldots,\left(u_{n}, v_{n}\right) ; h\right)
	\end{align*}
	
	\begin{enumerate}
		\setcounter{enumi}{1}
		\item Multilinear: For $i=1, \ldots, n$,
		\begin{align*}
			&g\left(\left(u_{1}, v_{1}\right), \ldots, a\left(u_{i}, v_{i}\right)+b\left(x_{i}, w_{i}\right), \ldots,\left(u_{n}, v_{n}\right) ; h\right)\\
			&=a g\left(\left(u_{1}, v_{1}\right), \ldots,\left(u_{i}, v_{i}\right), \ldots,\left(u_{n}, v_{n}\right) ; h\right)+b g\left(\left(u_{1}, v_{1}\right), \ldots,\left(x_{i}, w_{i}\right), \ldots,\left(u_{n}, v_{n}\right) ; h\right)
		\end{align*}
		
		\item $h$--$\gamma$ Diagonal:
	\end{enumerate}
	\begin{align*}
		g\left(\left(\gamma_{1}(t), \gamma_{2}(t)\right),\left(\gamma_{1}(t-h), \gamma_{2}(t-h)\right), \ldots\left(\gamma_{1}(t-(n-1)h), \gamma_{2}(t-(n-1)h)\right);h\right)=G(t).
	\end{align*}
\end{definition}
\begin{remark}
	The $h$--$\gamma$ blossom 
	$g\big((u_{1},v_{1}),\dots,(u_{n},v_{n});h\big)$
	is continuous. This follows since this blossom is symmetric and multilinear in its arguments, and every multilinear map is continuous. Along the $h$--$\gamma$ diagonal $g$ reduces to $G(t)$, which is continuous because $G \in \pi_{n}(\gamma_{1},\gamma_{2})$ and $\gamma_{1}$, $\gamma_{2}$ are continuous
	functions.
	
\end{remark}

To establish the existence and uniqueness of the $h$--$\gamma$ blossom, we shall use a different basis for $\pi_{n}(\gamma_{1},\gamma_{2})$ more compatible with the $h$--$\gamma$ diagonal property. For a fixed value of $h$, let $\mathcal{G}=\text{span}\left\{G_{n,k}(t;h)\right\}^{n}_{k=0}$, where
\begin{align}\label{hbasis}
	G_{n,k}(t;h)=\sum \prod_{j=1}^{k} \gamma_1\left(t-\left(i_{j}-1\right) h\right) \prod_{l=k+1}^{n} \gamma_2\left(t-\left(i_{l}-1\right) h\right)
\end{align}
and the sum is taken over all subsets $\left\{i_{1}, \ldots, i_{k}\right\}$ of $\{1, \ldots, n\}$.
Define 
\begin{align}
	d(t,x)=\gamma_{1}(t)\gamma_{2}(x)-\gamma_{2}(t)\gamma_{1}(x),
\end{align}
\begin{align}\label{basefunc}
\left(d(t,x)\right)^{n}_{h}=\prod_{j=0}^{n-1} d(t-jh,x)=	(\gamma_1(t)\gamma_2(x) - \gamma_2(t)\gamma_1(x))_h^n.
\end{align}
From \eqref{hbasis} and \eqref{basefunc}, we find that
\begin{equation}\label{rgenel}
	(\gamma_1(t)\gamma_2(x) - \gamma_2(t)\gamma_1(x))_h^n = \sum_{k=0}^{n} (-1)^{n-k} G_{n,k}(t;h) \gamma_2(x)^k \gamma_1(x)^{n-k}.
\end{equation}

In Appendix A, we will show that  for most values of $h$ the functions $\left\{G_{n,k}(t;h)\right\}_{k=0}^{n}$ are linearly independent.  The following example shows that the linear independence of the
	functions $\left\{G_{n,k}(t;h)\right\}_{k=0}^{n}$ may fail for certain values of $h$, 
even when $\gamma_1$ and $\gamma_2$ themselves are continuous, linearly independent, and translation invariant.

\begin{exmp}
	Consider the trigonometric functions $\gamma_{1}(x)=\cos x$ and $\gamma_{2}(x)=\sin x$ from Example \ref{e2}. For $n=2$, using (\ref{hbasis}), we obtain
	\begin{align*}
		G_{2,0}(t;h)&= 	\sin(t-h)\sin t = \cos h\,\sin^2 t - \sin h\,\sin t\cos t,\\
		G_{2,1}(t;h)&=\sin(t-h)\cos t + \sin t\cos(t-h)
		= 2\cos h\,\sin t\cos t + \sin h \hspace{0.05em}(\sin^2 t - \cos^2 t),\\
		G_{2,2}(t;h)&=\cos(t-h)\cos t = \cos h\,\cos^2 t + \sin h\,\sin t\cos t.
	\end{align*}
	Therefore, 
	\[
	\begin{pmatrix}G_{2,0}(t;h)\\ G_{2,1}(t;h)\\G_{2,2}(t;h)\end{pmatrix}
	= M \begin{pmatrix}\sin^2 t\\\sin t \cos t \\\cos^2 t\end{pmatrix}, \quad 	M =
	\begin{pmatrix} \cos h & -\sin h &0  \\\sin h & 2 \cos h &-\sin h\\     0 &\sin h& \cos h\end{pmatrix} .
	\]
	Then, by the linear independence of the functions $\gamma_1^k\gamma_{2}^{2-k}$, $k=0,1,2$, the functions $G_k(t;h)$ are linearly independent if and only if
	\[
	\det M = 2 \cos h\,(\cos^2 h+\sin^2 h)= 2\cos h \neq 0,
	\]
	which is so if and only if  $h\neq \frac{\pi}{2}+k\pi$, $k \in \mathbb{Z}$.
\end{exmp}
%
%
From here on we consider only those values of $h$ for which
the functions $\left\{G_{n,k}(t;h)\right\}_{k=0}^{n}$ are linearly independent (see Appendix A).
\begin{lemma}\label{lemmabasis}
For every fixed value of $h$, the functions	$\left\{G_{n,k}(t;h)\right\}^{n}_{k=0}$ are a basis for the space $\pi_{n}(\gamma_{1},\gamma_{2})$.
\end{lemma} 
\begin{proof}
	Since, the functions $G_{n,k}(t;h)$, $k = 0, \dots, n$, are linearly independent  it suffices to verify that each
	$G_{n,k}(t;h) \in \pi_n(\gamma_1, \gamma_2)$.  By translation invariance,
	\begin{align*}
	\begin{pmatrix}
		\gamma_1(t - mh) \\
		\gamma_2(t - mh)
	\end{pmatrix}
	= C^m
	\begin{pmatrix}
		\gamma_1(t) \\
		\gamma_2(t)
	\end{pmatrix}, \quad m\in \mathbb{N}_0.
		\end{align*}
		This relation and (\ref{hbasis}) show that $G_{n,k} \in \pi_n(\gamma_1, \gamma_2)$, $k=0, \ldots,n$.

\end{proof}

\begin{theorem}\label{theorem:blossom}
	Let $G\in \pi_{n}(\gamma_{1},\gamma_{2})$. Then there exists a unique function  $g\left(\left(u_{1}, v_{1}\right), \ldots,\left(u_{n}, v_{n}\right); h\right)$  that is symmetric, multilinear, and reduces to $G$ along $h$--$\gamma$ diagonal.
\end{theorem} 
\begin{proof}
To prove existence, let $G\in \pi_{n}(\gamma_{1},\gamma_{2})$. By Lemma \ref{lemmabasis}, $G(t)=\sum_{k=0}^{n} c_{k} G_{n,k}(t;h)$. Let
\begin{align}\label{blossom}
	g\left(\left(u_{1}, v_{1}\right), \ldots,\left(u_{n}, v_{n}\right); h\right)=\sum_{k=0}^{n} c_{k} \sum u_{i_{1}} \ldots u_{i_{k}} v_{i_{k+1}} \ldots v_{i_{n}},
\end{align}
 where the inside sum is taken over all subsets $\left\{i_{1}, \ldots, i_{k}\right\}$ of $\{1, \ldots, n\}$. Then $g$ is an $h$--$\gamma$ blossom of $G(t)$ because $g\left(\left(u_{1}, v_{1}\right), \ldots,\left(u_{n}, v_{n}\right); h\right)$ is symmetric, multilinear, and along the $h$--$\gamma$ diagonal
 \begin{align*}
 &	g\left(\left(\gamma_{1}(t), \gamma_{2}(t)\right),\left(\gamma_{1}(t-h), \gamma_{2}(t-h)\right), \ldots,\left(\gamma_{1}(t-(n-1)h), \gamma_{2}(t-(n-1)h)\right);h\right)\\&=\sum_{k=0}^{n} c_{k} \sum \prod_{j=1}^{k} \gamma_1\left(t-\left(i_{j}-1\right) h\right) \prod_{l=k+1}^{n} \gamma_2\left(t-\left(i_{l}-1\right) h\right)\\&=\sum_{k=0}^{n} c_{k} G_{n,k}(t;h).
 \end{align*}
 To prove uniqueness, suppose that $g_1$ and $g_2$ are two distinct blossoms of $G$. Then since the homogeneous elementary symmetric functions in $n$ variables $\left\{s_{n,k}\right\}_{k=0}^{n}$ are a basis for the symmetric multilinear functions in $n$ variables, there are the constants $\{c_k \}$ and $\{d_k \}$, such that $g_1=\sum_{k=0}^{n}c_k s_k $ and $ g_2=\sum_{k=0}^{n}d_k s_k$. Evaluating $g_1$ and $g_2$ along the $h$--$\gamma$ diagonal gives
$ 	G=\sum_{k=0}^{n}c_k G_{n,k}=\sum_{k=0}^{n}d_k G_{n,k}.$
 Since $\left\{G_{n,k}\right\}_{k=0}^{n}$ are linearly independent, $c_k=d_k$, $k=0,\ldots,n$. Therefore the blossom of $G$ is unique.
\end{proof}
\begin{exmp}\label{polyblossom}
	Let $\gamma_{1} = 1$, $\gamma_{2} = x$. Then for all $n$, we have 
	\(1 \in \pi_{n}(\gamma_{1}, \gamma_{2})\).
	The $h$--$\gamma$ blossom of the function $G(x) = 1$ in the space 
	\(\pi_{n}(\gamma_{1}, \gamma_{2})\) is the function $
	g_{n}\big( (u_{1},v_{1}), \ldots, (u_{n},v_{n}); h \big)
	= u_{1} \cdots u_{n}$.
\end{exmp}

\begin{exmp}\label{expblossom1}
		Let $\gamma_{1}(x)=\cos x$, $\gamma_{2}(x)=\sin x$, and $n$ be even. Here we will find the $h$-$\gamma$ blossom of the function $G(x) = 1$ in the space $\pi_{n}\left(\cos x, \sin x\right)$. Observe that 
		\begin{align*}
			1=(\cos^2 x+\sin^2 x )^{n/2} \in \pi_{n}\left(\cos x,\sin x \right).
		\end{align*}
		Therefore, by Theorem \ref{theorem:blossom}, the function $G(x)=1$ has an $h$--$\gamma$ blossom for all even dimensional spaces.
		Assume that $h\in\mathbb{R}$ is such that
		\[
		\sum_{P\in\mathcal{P}_{n}}\;\prod_{(i,j)\in P}\cos\big((i-j)h\big)\;\neq\;0,
		\]
		where $\mathcal{P}_{n}$ is the set of all pairings of $\{1,\dots,n\}$ -- that is, $\mathcal{P}_{n}$ is the set of all ways to group the numbers 
		$1,\ldots,n$ into $n/2$ separate pairs, with every index appearing in exactly one pair.  Then, the $h$--$\gamma$ blossom of the function $G(x)=1$  has the form
		\begin{align}\label{trigblossom}
			g\big((u_1,v_1),\dots,(u_{n},v_{n});h\big)
			=
			c_n(h)\,
			\sum_{P\in\mathcal{P}_{n}}
			\;\prod_{(i,j)\in P}\bigl(u_i u_j + v_i v_j\bigr),
		\end{align}
		where 
		\begin{align} \label{c(h)-Ex8}
			c_n(h)
			=\left(
			\sum_{P\in\mathcal{P}_{n}}
			\;\prod_{(i,j)\in P}\cos\bigl((i-j)h\bigr)\right)^{-1}.
		\end{align}
		The $h$--$\gamma$ blossom  (\ref{trigblossom}) is  symmetric in its arguments because it is defined as a sum over all pairings of the index set, and permuting the input pairs $(u_i,v_i)$ merely permutes the terms of this sum without changing its value. 
		The $h$--$\gamma$  blossom (\ref{trigblossom}) is also multilinear because for each pairing $P$, a given argument $(u_\ell,v_\ell)$ appears only in a single factor of the corresponding product of the form
	$u_\ell u_j + v_\ell v_j$. Hence the blossom is linear in  each of its $n$ arguments.	To verify the $h$--$\gamma$ diagonal property, for any pair $(i,j)$ we compute the product:
		\[
		\begin{aligned}
			u_i u_j + v_i v_j
			=
			\cos\!\bigl(t-(i-1)h\bigr)\cos\!\bigl(t-(j-1)h\bigr)
			+\sin\!\bigl(t-(i-1)h\bigr)\sin\!\bigl(t-(j-1)h\bigr)
			=  \cos\bigl((i-j)h\bigr).
		\end{aligned}
		\]
 Therefore  (\ref{trigblossom}) becomes
		\[
		g\bigl(\Gamma(t),\Gamma(t-h),\dots,\Gamma(t-(n-1)h);h\bigr)
		=
		c_n(h)
		\sum_{P\in\mathcal{P}_{n}}
		\;\prod_{(i,j)\in P}\cos((i-j)h)=1.
		\]
		
	To help illustrate and clarify this general formula, we compute the $h$--$\gamma$ blossom of the
		constant function $G(x)=1$ in the cases $n=2$ and $n=4$.\\
		\noindent\textbf{Case $\boldsymbol{n=2}$.}
	 There is only one  pairing, $
		\mathcal{P}_2 = \bigl\{ (1,2) \bigr\}.
		$
		Substituting this pairing into the general formula gives
		\[
		g_2\bigl((u_1,v_1),(u_2,v_2);h\bigr)
		= c_2(h)\,\bigl(u_1u_2 + v_1v_2\bigr),\,\,\,\,
		c_2(h)
		=
		\left(\cos\bigl((1-2)h\bigr)\right)^{-1}
		= \sec h.
		\]
		\noindent\textbf{Case $\boldsymbol{n=4}$.}
		Now, the set of pairings is
		\[
		\mathcal{P}_4
		=
		\bigl\{
		\{(1,2),(3,4)\},
		\{(1,3),(2,4)\},
		\{(1,4),(2,3)\}
		\bigr\}.
		\]
		Inserting these pairings into (\ref{trigblossom}) yields
		\[
		\begin{aligned}
			g_4\bigl((u_1,v_1),(u_2,v_2),(u_3,v_3),(u_4,v_4);h\bigr)
			&=
			c_4(h)\Big(
			(u_1u_2+v_1v_2)(u_3u_4+v_3v_4)
			\\
			&\qquad
			+(u_1u_3+v_1v_3)(u_2u_4+v_2v_4)
			\\
			&\qquad
			+(u_1u_4+v_1v_4)(u_2u_3+v_2v_3)
			\Big),
		\end{aligned}
		\]
		where
		\[
			c_4(h)
			=\left(\cos^2 h + \cos^2(2h) + \cos(3h)\cos h\right)^{-1}.
		\]
\end{exmp}

In addition to the classical trigonometric case treated in Example~\ref{expblossom1}, similar computations can be carried out for the discrete trigonometric, hyperbolic, and discrete hyperbolic functions. Since these functions satisfy similar identities, analogous formulas for the blossom of the constant function $G(x)=1$ can be derived in exactly the same manner for these functions.

	\begin{exmp}\label{expd}
	For fixed $x$, here we will find the $h$--$\gamma$ blossom of $G(t)=(d(t, x))_{h}^{n}$. Recall that by  (\ref{basefunc}), (\ref{rgenel}), and Lemma \ref{lemmabasis},  $(d(t,x))_h^{n}\in \pi_n(\gamma_{1}, \gamma_{2})$. Hence
	$G(t)=(d(t, x))_{h}^{n}$ has an $h$--$\gamma$ blossom $g\left(\left(u_{1}, v_{1}\right), \ldots,\left(u_{n}, v_{n}\right) ; h\right)$. 
	Moreover
	\begin{align}
		g\left(\left(u_{1}, v_{1}\right), \ldots,\left(u_{n}, v_{n}\right) ; h\right)=\prod_{k=1}^{n}\left( \gamma_{2}(x) u_{k}-\gamma_{1}(x) v_{k}\right)
	\end{align}
	since the right-hand side is symmetric, multilinear, and reduces to
	$G(t)=\left(d(t, x)\right)_{h}^{n}$ along the $h$--$\gamma$ diagonal.
\end{exmp}
\section{Recursive Evaluation Algorithms}\label{section4}

From the definitions of $d$ and $\Gamma$, it follows immediately that
\begin{align*}
	d(u,v) \Gamma(w)+	d(v,w) \Gamma(u)+	d(w,u) \Gamma(v)=(0,0),
\end{align*}
which can be written in the form
\begin{align}\label{uvm-Id}
	\frac{d(u,v)}{d(w,v)} \Gamma(w)+	\frac{d(w,u)}{d(w,v)} \Gamma(v)=	 \Gamma(u).
\end{align}
Based on equation \eqref{uvm-Id}, next we present an algorithm
for computing arbitrary  $h$--$\gamma$ blossom values from a set of special  $h$--$\gamma$ blossom values that will appear later in the dual functional property  (Theorem \ref{dualfuncprop}).

\begin{theorem}\label{recursive}
Let $a,b$ be arbitrary constants such that $d(a-jh, b-ih)\neq 0$ for $0\leq i\leq j\leq n-1$, and let $G \in \pi_{n}\left(\gamma_{1}, \gamma_{2}\right)$ with  $h$--$\gamma$ blossom $g$. Set
	\begin{equation}\label{h-gammablossom}
		Q_{i}^{0}=g(\Gamma(a-i h), \ldots, \Gamma(a-(n-1) h), \Gamma(b), \Gamma(b-h), \ldots, \Gamma(b-(i-1) h) ; h), 
	\end{equation}
	$i=0, \ldots, n$ and define recursively the set of multilinear functions
	\begin{align*}
		Q_{i}^{k+1}(\Gamma(u_1),\ldots,\Gamma(u_{k+1});h)=&\frac{d(u_{k+1},b-ih)}{d(a-(i+k)h,b-ih)} Q_{i}^{k}(\Gamma(u_1),\ldots,\Gamma(u_k);h)\\&+\frac{d(a-(i+k) h,u_{k+1})}{d(a-(i+k)h,b-ih)} Q_{i+1}^{k}(\Gamma(u_1),\ldots,\Gamma(u_k);h)
	\end{align*}
	for $i=0, \ldots, n-k-1$ and $k=0, \ldots, n-1$.
	Then
	\begin{align}\nonumber
		&Q_{i}^{k}(\Gamma(u_1),\ldots,\Gamma(u_k);h)\\ \label{RRalg}&=g\left(\Gamma(a-(k+i) h), \ldots, \Gamma(a-(n-1) h), \Gamma(b), \Gamma(b-h), \ldots, \Gamma(b-(i-1) h),  \Gamma(u_{1}), \ldots,  \Gamma(u_{k}) ; h\right), 
	\end{align}
	$i=0, \ldots, n-k,\,\, k=0, \ldots, n$. In particular,
	$$
	Q_{0}^{n}(\Gamma(u_1),\ldots,\Gamma(u_n);h)=g\left( \Gamma(u_{1}), \ldots,  \Gamma(u_{n}) ; h\right).
	$$
\end{theorem} 
\begin{proof}
 This result follows by induction on $k$ from the symmetry and multilinearity of the  $h$--$\gamma$ blossom and \eqref{uvm-Id}. The case $n=3$ is illustrated in Figure \ref{fig1}.	
 \end{proof}
\begin{center}
	\begin{figure}[h!]
		\centering
		\includegraphics[width=1.02 \linewidth]{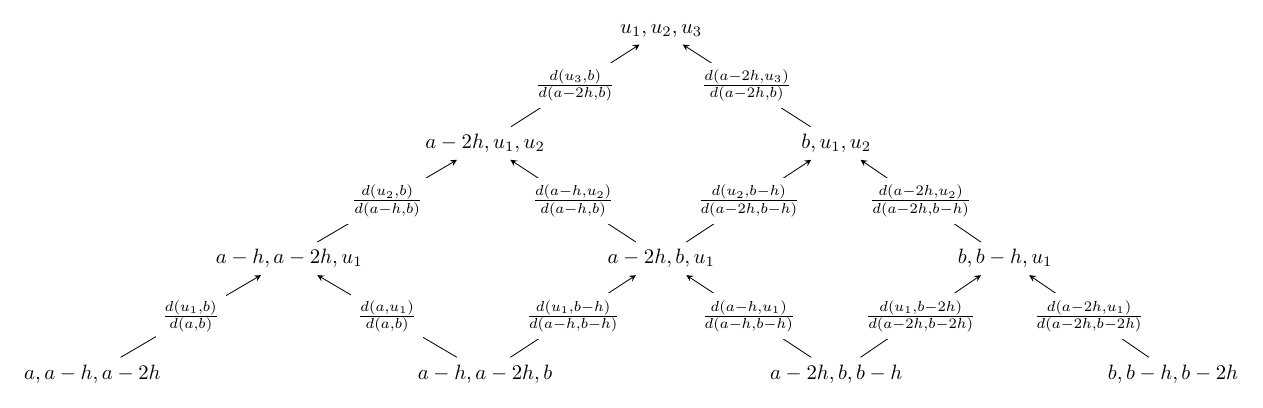}
		\caption{Computing $g(\Gamma(u_1),\ldots,\Gamma(u_n);h)$ from the initial $h$--$\,\gamma$ blossom values $ Q^0_i$, $i = 0,\ldots ,n$. Here we illustrate the case $n=3$ and we use the notation $u,v,w$ to represent the blossom value $g(\Gamma(u), \Gamma(v), \Gamma(w);h)$.}
		\label{fig1}
	\end{figure}
\end{center}
\begin{theorem}[Recursive Evaluation Algorithm]\label{recursivex}
	Let $G \in \pi_{n}\left(\gamma_{1}, \gamma_{2}\right)$ and let $g\left(\left(u_{1}, v_{1}\right), \ldots,\left(u_{n}, v_{n}\right); h\right)$ be the  $h$--$\gamma$ blossom of $G(x)$. For \(h\neq 0\), there are $n!$ recursive evaluation algorithms for $G(x)$ defined as follows: Let $a,b$ be arbitrary constants such that $d(a-jh, b-ih)\neq 0$ for $0\leq i\leq j\leq n-1$, and let  $\sigma$ be a permutation of $\{1, \ldots, n\}$. Set
	
	$$
	P_{i}^{0}=g(\Gamma(a-i h), \ldots,\Gamma( a-(n-1) h), \Gamma(b), \Gamma(b-h), \ldots, \Gamma(b-(i-1) h) ; h), \quad i=0, \ldots, n
	$$
	and define recursively
	\begin{align*}
		P_{i}^{k+1}(x)=\frac{d(x-(\sigma(k+1)-1
			)h,b-ih)}{d(a-(i+k)h,b-ih)} P_{i}^{k}(x)+\frac{d(a-(i+k) h,x-(\sigma(k+1)-1
			)h)}{d(a-(i+k)h,b-ih)} P_{i+1}^{k}(x) 
	\end{align*}
	for $i=0, \ldots, n-k-1$ and $k=0, \ldots, n-1$. 
	Then
	\begin{align}\nonumber
		P_{i}^{k}(x)=\, & g(\Gamma(a-(k+i) h), \ldots, \Gamma(a-(n-1) h), \Gamma(b), \Gamma(b-h), \ldots,  \Gamma(b-(i-1) h),\\ \label{Pikx-sig}
		& \Gamma(x-(\sigma(1)-1) h), \ldots,\Gamma( x-(\sigma(k)-1) h) ; h),
	\end{align}
	$i=0, \ldots, n-k, k=0, \ldots, n$. In particular,
	
	\begin{equation*}
		P_{0}^{n}(x)=g(\Gamma(x-(\sigma(1)-1) h), \ldots, \Gamma(x-(\sigma(n)-1) h ); h)=G(x) . 
	\end{equation*}
\end{theorem}

\begin{proof}
	The result follows from Theorem \ref{recursive} by substituting the specific values for the  $h$--$\gamma$ blossom parameters: $u_{i}=x-(\sigma(i)-1) h,\, i=1, \ldots, n$. The case $n=3$ for $\sigma_{i}=i$ is illustrated in Figure \ref{fig2}.
\end{proof}
\begin{center}
	\begin{figure}[h!]
		\centering
		\includegraphics[width=1.02 \linewidth]{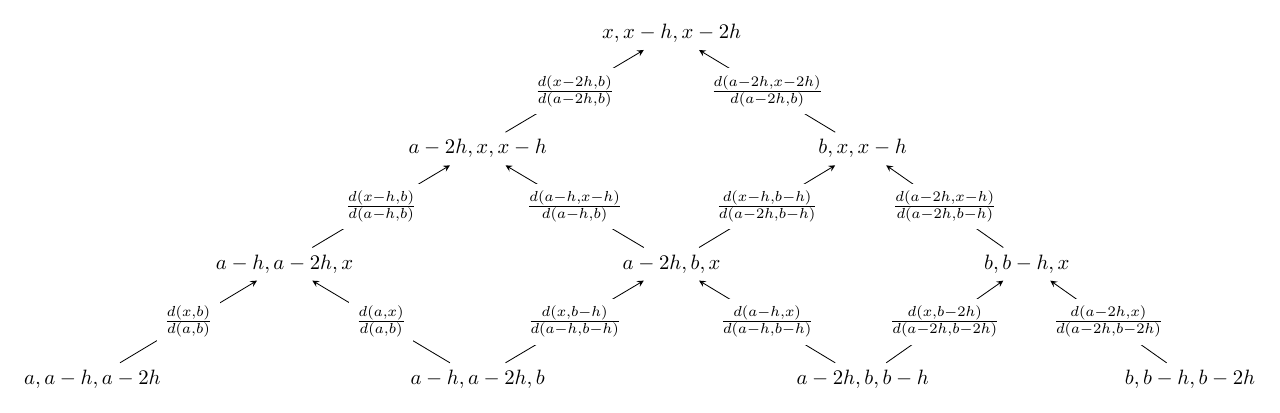}
		\caption{Recursive evaluation algorithm for $G(x)$ for $\sigma(i)=i$. Here we illustrate the case $n=3$ and we use the notation $u,v,w$ to represent the blossom value $g(\Gamma(u), \Gamma(v), \Gamma(w);h)$. }
		\label{fig2}
	\end{figure}
\end{center}
\section{Bernstein Basis Functions} \label{section5}
In this section, the notation $[a, b]$ denotes the interval determined by the endpoints $a$ and $b$, regardless of their order. Thus, $x \in [a, b]$ means that $x$ lies between $a$ and $b$, i.e., $\min\{a, b\} \le x \le \max\{a, b\}$.
We are now going to define an $h$--$\gamma$ version of Bernstein basis functions and Bernstein B\'{e}zier curves for the space $\pi_{n}\left(\gamma_{1}, \gamma_{2}\right)$ over the interval $[a,b]$. In the recursive evaluation algorithm, we start from the base row consisting of control points $P^0_k$, $k=0,\ldots,n$. 
At each stage, the values are computed by recursive interpolation, so that every intermediate point $P^m_j(x)$ is expressed 
as a linear combination of the control points $P^0_k$. When we reach the apex, the function value $G(x)$ emerges as a linear 
combination of the control points with certain coefficient functions. These coefficient functions, attached to the base control 
points $P^0_k$, are by definition the $h$--$\gamma$ Bernstein basis functions which we will denote by $
B^n_k(x,[a,b];\gamma,h)
$. 
 \begin{remark}
 	Let $G \in \pi_{n}(\gamma_{1},\gamma_{2})$ with  $h$--$\gamma$ blossom $g$. By Theorem~\ref{recursivex}, there exist $n!$ recursive evaluation algorithms parameterized by permutations $\sigma$ of $\{1,\ldots,n\}$. Each choice of $\sigma$ specifies a different order for inserting the arguments 
 $
 	x, \; x-h, \; \ldots, \; x-(n-1)h$
 	into the blossom. Although every such algorithm evaluates to the same function value $G(x)$ at the apex, the coefficient functions attached to the initial control points vary with the choice of permutation, yielding different families of  $h$--$\gamma$ Bernstein basis functions.
 	
 	Figures \ref{Fig4} and \ref{Fig5} display the recursive evaluation algorithm for $n=2$ with different choices of the permutation $\sigma$. Figure~\ref{Fig4} corresponds to the identity permutation $\sigma(i)=i$, while Figure~\ref{Fig5} corresponds to the reverse ordering $\sigma(i)=n+1-i$. Both diagrams give rise to valid recursive evaluation algorithms and thus to distinct representations of the same function $G(x)$. But the explicit formulas for the $h$--$\gamma$ Bernstein basis functions may differ.  
 	Equations~\eqref{bbernstein1h} and~\eqref{bbernstein1h2} show this difference: the expressions for $B^2_{0}(x,[a,b];\gamma,h) $ and $B^2_{2}(x,[a,b];\gamma,h) $ are identical, 
 	while the middle basis functions $B^2_{1}(x,[a,b];\gamma,h) $ differ depending on $\sigma$. 
 	To highlight this effect, we have written the altered terms in bold. 
 	\begin{figure}[!htb]
 		\begin{minipage}{0.49\textwidth}
 			\centering
 			\includegraphics[width=1\linewidth]{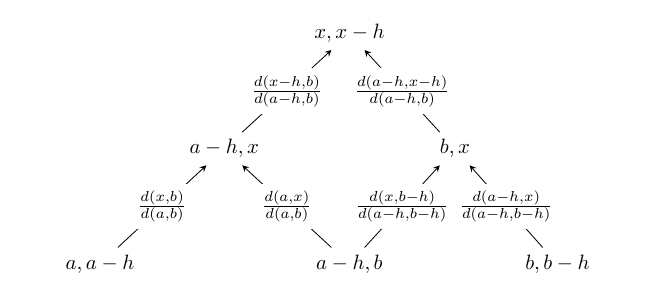}
 			\caption{Recursive evaluation algorithm for $G(x)$ with $\sigma(i)=i$. Here we illustrate the case $n=2$ and we use the notation $u,v$ to represent the blossom value $g(\Gamma(u), \Gamma(v);h)$. }\label{Fig4}
 		\end{minipage}\hfill
 		\begin{minipage}{0.49\textwidth}
 			\centering
 			\includegraphics[width=1\linewidth]{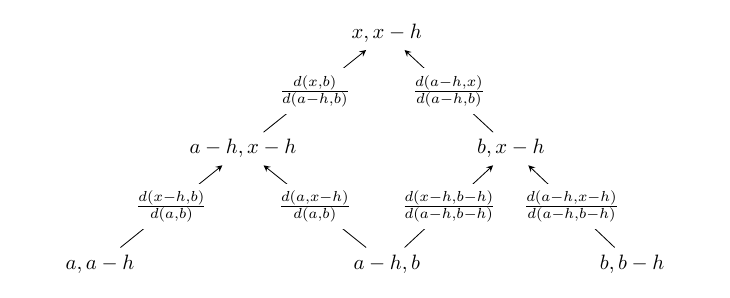}
 			\caption{Recursive evaluation algorithm for $G(x)$ with $\sigma(i)=n+1-i$. Here we illustrate the case $n=2$ and we use the notation $u,v$ to represent the blossom value $g(\Gamma(u), \Gamma(v);h)$.}\label{Fig5}
 		\end{minipage}
 	\end{figure}
 	For  $\sigma(i)=i$ and $n=2$, from Figure \ref{Fig4} we get
 	\begin{equation}\label{bbernstein1h}
 		\begin{aligned}
 			& B_{0}^{2}(x,[a,b];\gamma,h) =\frac{d(x, b) d(x-h, b)}{d(a, b) d(a-h, b)}, \\
 			& B_{1}^{2}(x,[a,b];\gamma,h) =\frac{\boldsymbol{d(a, x)d(x-h, b)}}{d(a, b) d(a-h, b)} +\frac{\boldsymbol{d(x, b-h)d(a-h, x-h)}}{d(a-h, b-h)d(a-h, b)},\\
 			& B_{2}^{2}(x,[a,b];\gamma,h) =\frac{ d(a-h, x)d(a-h, x-h)}{d(a-h, b-h) d(a-h, b)},\end{aligned}
 	\end{equation}
 	whereas for $\sigma(i)=n+1-i$ and $n=2$, from Figure \ref{Fig5}, we get
 	\begin{equation}\label{bbernstein1h2}
 		\begin{aligned}
 			& B_{0}^{2}(x,[a,b];\gamma,h) =\frac{d(x-h, b)d(x, b) }{d(a, b) d(a-h, b)}, \\
 			& B_{1}^{2}(x,[a,b];\gamma,h) =\frac{\boldsymbol{d(a, x-h)d(x, b)}}{d(a, b) d(a-h, b)} +\frac{\boldsymbol{d(x-h, b-h) d(a-h, x)}}{d(a-h, b-h)d(a-h, b)},\\
 			& B_{2}^{2}(x,[a,b];\gamma,h) =\frac{d(a-h, x-h) d(a-h, x)}{d(a-h, b-h) d(a-h, b)}.\end{aligned}
 	\end{equation}
 We will use the natural choice $\sigma(i)=i$, corresponding to inserting the arguments in the order $x, x-h, \dots, x-(n-1)h$. With this ordering, the recursive evaluation algorithm (see Figure \ref{fig2}) becomes straightforward, and more importantly, provides a subdivision algorithm directly analogous to the classical de Casteljau subdivision algorithm for B{\'e}zier curves (see Theorem \ref{subdivision}). Hence, fixing $\sigma(i)=i$ allows us to develop a convenient recursive structure for the $h$--$\gamma$ Bernstein basis functions and the associated subdivision algorithm.

 \end{remark}
\begin{theorem}\label{theorem:basis}
The  $h$--$\gamma$ Bernstein basis functions  $B^{\,n}_k(x,[a,b];\gamma,h)$  satisfy the following recurrence:
	\begin{align}\nonumber
	\,B_{0}^{0}(x,[a,b] ;\gamma, h)=1,\\ \nonumber
	B_{k}^{n}(x,[a,b] ;\gamma, h)=\,&\frac{d(a-(k-1)h,x)}{d(a-(k-1)h,b-(k-1)h)}\, B_{k-1}^{n-1}(x-h,[a-h,b] ;\gamma, h)\\&+\frac{d(x,b-kh)}{d(a-kh,b-kh)} B_{k}^{n-1}(x-h,[a-h,b] ;\gamma, h)\label{recurrenceBB}
\end{align}
$k=0, \ldots, n$, where $B_{-1}^{n-1}(x,[a,b] ;\gamma, h)=0$ and $B_{n}^{n-1}(x,[a,b] ; \gamma, h)=0$.
\end{theorem}
\begin{proof}	The proof is by induction on $n$. By Theorem \ref{recursivex} for the permutation $\sigma(i)=i$ and $n=1$, we have
	\begin{equation}\label{eq:thm6-n1}
		P^{1}_0(x)
		=\frac{d(x,b)}{d(a,b)}\,P^0_0
		\;+\;
		\frac{d(a,x)}{d(a,b)}\,P^0_1 .
	\end{equation}
Therefore by definition, the coefficients of $P^0_0$ and $P^0_1$ in (\ref{eq:thm6-n1})
	are the basis functions of order 1:
	\begin{align*}
		B^{1}_{0}(x,[a,b];\gamma,h) = \frac{d(x,b)}{d(a,b)},\quad
		B^{1}_{1}(x,[a,b];\gamma,h) = \frac{d(a,x)}{d(a,b)},
	\end{align*}
which is  (\ref{recurrenceBB}) for $n=1$. Suppose that (\ref{recurrenceBB}) holds for order $n-1$  on the interval $[a-h,b]$. We will prove that  (\ref{recurrenceBB}) holds for order $n$. Again consider the recursive evaluation algorithm for the permutation $\sigma(i)=i$. In this scheme every node at level $k+1$ is formed from two nodes at level $k$ by interpolation with weights expressed using $d(\cdot,\cdot)$. Placing $1$ at the apex and propagating downward, the value attached to each base control point $P^0_k$
is precisely the basis function $B^{\,n}_k\!\left(x,[a,b];\gamma,h\right)$. By Theorem (\ref{recursivex}), for a fixed $k$, the point $P^0_k$ contributes to two nodes at level one:
\begin{align}
	P^{1}_i(x)
	&=\underbrace{\frac{d(x,\,b-ih)}{d(a-ih,\,b-ih)}}_{\displaystyle \alpha_i(x)}\,P^0_i
	\;+\;
	\underbrace{\frac{d(a-ih,\,x)}{d(a-ih,\,b-ih)}}_{\displaystyle \beta_{i}(x)}\,P^0_{i+1}, \quad i=k-1, k. \label{eq:right-edge}
	\end{align}
\noindent Hence the coefficients of $P^0_k$ on the two incident edges are $\alpha_{k}(x)$ and $\beta_{k-1}(x)$  in \eqref{eq:right-edge}. Using the induction hypothesis on $[a-h,b]$ at $x-h$, we obtain, for $n\ge1$ and $0\le k\le n$,
\begin{equation}\label{eq:hB-rec}
	B_k^{\,n}(x,[a,b];\gamma,h)
	=\beta_{k-1}(x)\,B_{k-1}^{\,n-1}(x-h,[a-h,b];\gamma,h)
	+\alpha_{k}(x)\,B_{k}^{\,n-1}(x-h,[a-h,b];\gamma,h),
\end{equation}
which gives the result. For an illustration of the case $n=3$, see Figure	\ref{fig3}.
\end{proof}
As in the classical case, when we put $1$ at the apex in the diagram of the recursive evaluation algorithm in Figure \ref{fig2}, the values in the triangle propagate downwards by the same recursive interpolation, and what we get at the base  are exactly the $h$-Bernstein basis functions, as depicted in Figure \ref{fig3}.
\begin{center}
	\begin{figure}[ht!]
		\centering
		\includegraphics[width=1.01 \linewidth]{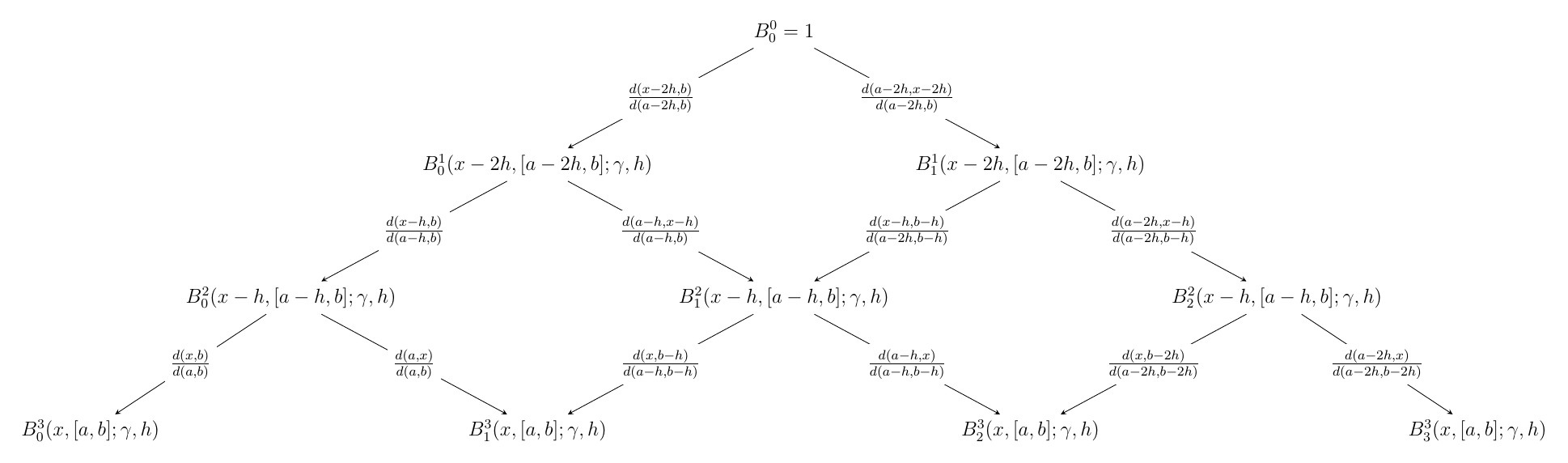}
		\caption{ Placing $1$ at the apex and propagating down yields the
			$h$--$\gamma$ Bernstein basis functions at the base.}
		\label{fig3}
	\end{figure}
\end{center}
\begin{theorem} 
	The $h$--$\gamma$ Bernstein basis functions are \textbf{shift-invariant}: For all $\Delta$ such that $C(\Delta)\neq 0$,
	\begin{align}\label{tinvariance}
		B_{k}^{n}(t+\Delta ;[a+\Delta, b+\Delta] ; \gamma, h)=B_{k}^{n}(t ;[a, b] ; \gamma, h).
	\end{align}
\end{theorem}
\begin{proof}
	Notice that $d(u,v)= \gamma_{1}(u) \gamma_{2}(v)-\gamma_{2}(u) \gamma_{1}(v)= \det\!\big[\,\big(\gamma_1(u),\gamma_2(u)\big)^T\;
	\big(\gamma_1(v),\gamma_2(v)\big)^T\,\big]$.  Hence by translation invariance
	\begin{align}\nonumber
		d(u-\Delta,v-\Delta) 
		&= \det\!\Big[\,\big(\gamma_1(u-\Delta),\gamma_2(u-\Delta)\big)^T\;
		\big(\gamma_1(v-\Delta),\gamma_2(v-\Delta)\big)^T\,\Big] \\ \nonumber
		&= \det\!\Big[\,C(\Delta)\big(\gamma_1(u),\gamma_2(u)\big)^T\;
		C(\Delta)\big(\gamma_1(v),\gamma_2(v)\big)^T\,\Big] \\ \nonumber
		&= \det(C(\Delta))\,
		\det\!\Big[\,\big(\gamma_1(u),\gamma_2(u)\big)^T\;
		\big(\gamma_1(v),\gamma_2(v)\big)^T\,\Big] \\ \label{dtrans}
		&= \det(C(\Delta))\,d(u,v).
	\end{align}
	By  (\ref{dtrans}) and Theorem \ref{theorem:basis}, the same constant  $\det(C(\Delta))$ multiplies every $d(.,.)$ for a fixed shift $\Delta$ and appears in every numerator and denominator in the recursive
	evaluation algorithm \eqref{recurrenceBB} for $B_{k}^{n}(t+\Delta ;[a+\Delta, b+\Delta] ; \gamma, h)$. Cancelling  $\det(C(\Delta))$ yields $B_{k}^{n}(t ;[a, b] ; \gamma, h)$.
\end{proof}

For  special choices of $\gamma_1$ and $\gamma_{2}$,  we get known $h$--$\gamma$ Bernstein basis functions. For  $\gamma_1=1$ and $\gamma_{2}=x$,  $\pi_{n}(\gamma_{1},\gamma_{2})$ is the space of polynomials of degree $n$
and $d(u, v) = v-u$. Thus \cite{ron1}
\begin{align}\label{hbernsteinbasis}
	B_{k}^{n}(x,[a,b]; h)=\binom{n}{k} \frac{\prod_{j=0}^{k-1}(x-a + jh)\prod_{j=0}^{n-k-1} (b-x+jh)}{\prod_{j=0}^{n-1}(b-a+jh)}.
\end{align}
 For $\gamma_1(x)=\cos x$, $\gamma_2(x)=\sin x$, we have $
d(u,v)= \sin(v - u)$. Therefore, by \eqref{bbernstein1h}, the explicit formulas for
the $h$--$\gamma$ Bernstein basis functions for the space $\pi_{2}(\cos(x), \sin(x))$ are
\begin{align*}
	B^2_0(x,[a,b];\gamma,h)
	&=\frac{\sin(b-x)\,\sin(b-x+h)}
	{\sin(b-a)\,\sin(b-a+h)},\\
	B^2_1(x,[a,b];\gamma,h)
	&=\frac{\sin(x-a)}{\sin(b-a)}
	\left(
	\frac{\sin(b-x+h)}{\sin(b-a+h)}
	+\frac{\sin(b-h-x)}{\sin(b-a+h)}
	\right),\\
	B^2_2(x,[a,b];\gamma,h)
	&=\frac{\sin(x-a)\,\sin(x-a+h)}
	{\sin(b-a)\,\sin(b-a+h)}.
\end{align*}
As $h \to 0$, each basis function tends to the corresponding classical trigonometric Bernstein basis function introduced in \cite{gonsor}. 

In the special case $h=0$, the $h$--$\gamma$ Bernstein basis functions reduce to the
$\gamma$--Bernstein basis functions for the space $\pi_{n}\left(\gamma_{1},\gamma_{2}\right)$ derived in \cite{cetin2}:
\begin{align}\label{gammaber}
	B_k^n(x,[a,b]; \gamma) =
	\binom{n}{k}
	\left( \frac{d(a,x)}{d(a,b)} \right)^k
	\left( \frac{d(x,b)}{d(a,b)} \right)^{\,n-k}.
\end{align}

\begin{definition}
	Let $\Gamma(t)=\left(\gamma_{1}(t), \gamma_{2}(t)\right),\, t\in[a,b]$, be a planar curve, where $d(a-jh, b-ih)\neq 0$ for $0\leq i\leq j\leq n-1$. The $h$--$\gamma$ B\'{e}zier curves of order $n$ on the planar parametric domain $\Gamma(t)$  are defined by
	\begin{align}
		G(x)=\sum_{k=0}^{n}b_kB_{k}^{n}(x,[a,b] ; \gamma, h), \quad x\in[a,b],
	\end{align}
where $b_k$, $k=0,\ldots, n$ are called the \textit{control points} of the $h$--$\gamma$ B\'{e}zier curve $G(x)$.
\end{definition}

\begin{theorem}\label{dual}
	 Every curve $G\in \pi_{n}(\gamma_{1},\gamma_{2})$ is an $h$--$\gamma$ B\'{e}zier curve of the form
	 \begin{align}\label{dualeqn}
	 G(x)=\sum_{k=0}^{n}  g(\Gamma(a-k h), \ldots,\Gamma( a-(n-1) h), \Gamma(b), \Gamma(b-h), \ldots, \Gamma(b-(k-1) h) ; h)B_{k}^{n}(x,[a,b];\gamma,h),
	 \end{align}
	where $g$ is the $h$--$\gamma$  blossom of $G$.\end{theorem}
	\begin{proof}
		Equation \eqref{dualeqn} follows from the recursive evaluation algorithm in Theorem~\ref{recursivex} for the permutation $\sigma(i)=i$, $i=1, \ldots, n$ and the definition of the
		$h$--$\gamma$ Bernstein basis functions $B^n_k(x,[a,b];\gamma,h)$ for that permutation.
	\end{proof}

	\begin{corollary}\label{basis}
		The $h$--$\gamma$ Bernstein basis functions of order $n$ over the interval $[a, b]$ form a basis for the space $\pi_{n}(\gamma_{1},\gamma_{2})$.
	\end{corollary}
	\begin{proof}
		By the construction of the $h$--$\gamma$ Bernstein basis functions from the recursive evaluation algorithm, each  $h$--$\gamma$ Bernstein function $B_{k}^{n}(x ;[a, b] ; \gamma, h)$  belongs to the space $\pi_{n}(\gamma_{1},\gamma_{2})$. By (\ref{dualeqn}), the functions  $\left\{B_{k}^{n}(x ;[a, b] ; \gamma, h)\right\}_{k=0}^{n}$ also span $\pi_{n}(\gamma_{1},\gamma_{2})$, which
		is a vector space of dimension $n + 1$. 
	\end{proof}
	\begin{corollary}\label{controlpoints}
		The control points of an  $h$--$\gamma$ Bernstein B\'{e}zier curve over the interval $[a, b]$ are unique.
	\end{corollary}
	\begin{proof}
		Since by Corollary \ref{basis}, the function $B_{k}^{n}(x ;[a, b] ; \gamma, h)$  form a basis for the space $\pi_n(\gamma_{1},\gamma_{2})$, the  $h$--$\gamma$ B\'{e}zier curves have unique coefficients in this basis. Therefore the control points are unique.
	\end{proof}
	\begin{theorem}[Dual Functional Property]\label{dualfuncprop}
		Let  $G\in \pi_{n}(\gamma_{1},\gamma_{2})$ and let $g$ be the $h$--$\gamma$ blossom of $G$. Then the Bernstein B\'{e}zier control points of G are given by
		\begin{align}\label{dualcontrol}
			b_k=g(\Gamma(a-k h), \ldots, \Gamma(a-(n-1) h), \Gamma(b), \Gamma(b-h), \ldots, \Gamma(b-(k-1) h) ; h),\quad k=0,\ldots, n.
		\end{align}
	\end{theorem}
	\begin{proof}
		This result follows  from Theorem \ref{dual}.
	\end{proof}
	  \begin{theorem}
		Let $G(x)$ be an  $h$--$\gamma$ B{\'e}zier curve of degree $n$ over the interval $[a,b]$ with control points $P_i^0$, $i=0,\ldots,n$.
		Let $P^k_i(x)$, $k=0,\ldots,n$, $i=0,\ldots,n-k$, be the nodes in the  $h$--$\gamma$ recursive evaluation algorithm for $G(x)$ for the identity permutation.
		Then
		\begin{equation}
			P^k_i(x)=\sum_{j=0}^{k} P_{i+j}^0\, B^{k}_{j}\!\left(x+i h;\,[a,b]; \gamma, h\right).
		\end{equation}
	\end{theorem}
	\begin{proof}
		By Theorems \ref{recursive}  and  \ref{recursivex}, the  $h$--$\gamma$ blossom of $P^k_i(t)$ is given by \eqref{RRalg}. The dual functional property for the interval $[c,d]$ yields
		\begin{align}\label{pp2}
			P^k_i(t)=\sum_{j=0}^{k}
			Q^k_i\!\big(\Gamma(c-jh),\ldots,\Gamma(c-(k-1)h),\Gamma(d),\Gamma(d-h),\ldots,\Gamma(d-(j-1)h);h\big)\
			B^{k}_{j}\!\left(t;\,[c,d]; \gamma, h\right)
		\end{align}
		Select $[c,d]=[a-ih,b-ih]$. By  shift invariance (\ref{tinvariance}), $B^{k}_{j}\!\left(t;\,[c,d]; \gamma, h\right)=	B^{k}_{j}\!\left(t+ih;\,[a,b]; \gamma, h\right)$. Now the result follows from (\ref{RRalg}), (\ref{pp2}),  and (\ref{h-gammablossom}).
	\end{proof}
			\subsection{Partitions of Unity and Marsden Identities}
			\begin{theorem}[Partition of Unity]
				Let $\gamma_{1} = 1$, $\gamma_{2} = x$. Then
				\[
				\sum_{k=0}^{n} B_{k}^{n}(x,[a,b];\gamma,h) = 1.
				\]
					\end{theorem}
		\begin{proof}
				This result follows immediately from Example \ref{polyblossom} and Theorem \ref{dualfuncprop}.
		\end{proof}

		\begin{theorem}[Partition of Unity for Trigonometric Functions]\label{the:unit}
			Let $\gamma_{1}=\cos x$, $\gamma_{2}=\sin x$ and let $n$ be an even number. Then $\sum_{k=0}^{n} b_k B_k^{n}(x,[a,b];\gamma,h)= 1$, where $$ b_k=	c_n(h)
			\sum_{P\in \mathcal{P}_{n}}
			\prod_{(i,j)\in P} \cos(t_{i,k}-t_{j,k}),$$
			$c_n(h)$ is defined by \eqref{c(h)-Ex8}  in Example \ref{expblossom1}, and
			
		\[
		\begin{alignedat}{2}
			t_{i,k} = a - (k + i - 1)h, &\qquad i &= 1,\ldots,n-k,\quad\;\;\,\\
			t_{i,k} = b - (i - n + k - 1)h, &\qquad i &\,= n-k+1,\ldots,n.
		\end{alignedat}
		\]
		\end{theorem}
			\begin{proof}
				This result follows immediately from Example~\ref{expblossom1} and Theorem \ref{dualfuncprop}.
				
			\end{proof}
		
		Theorem \ref{the:unit} remains valid if we replace $\cos x$ and $\sin x$ by $\cos_{d}x$ and $\sin_{d}x$; the proof is almost identical. Similar results also hold for both $\Gamma=\left(\cosh x, \sinh x\right)$ and $\Gamma=\left(\cosh_d x, \sinh_d x\right)$. Once again the proofs are much the same.
		\begin{theorem}[Marsden identity]\label{theorem:marsden}
			\begin{align}\label{MarsId}
				(d(t, x))_{h}^{n}=\sum_{k=0}^{n}\left\{ \prod_{j=0}^{k-1} d(b-j h, x) \prod_{j=k}^{n-1} d(a-j h, x)\right\} B_{k}^{n}(t,[a, b] ; \gamma, h).
			\end{align}
			
		\end{theorem}
		\begin{proof}
			This result follows immediately from Example \ref{expd} and Theorem \ref{dualfuncprop}. 
		\end{proof}
		
		As an example of Marsden's identity,  consider the case $\gamma_1(x)=\cos x,\;\gamma_2(x)=\sin x$. In this case,  $d(u,v) = \sin(v-u)$ and 
		\[
		\bigl(\sin(t-x)\bigr)^{n}_{h}
		=
		\sum_{k=0}^{n}
		\left\{
		\prod_{j=0}^{k-1} \sin\!\bigl(b-jh-x\bigr)
		\prod_{j=k}^{n-1} \sin\!\bigl(a-jh-x\bigr)
		\right\}
		B^{n}_{k}(t,[a,b];\gamma,h).
		\]
	For another example consider the
	case where  $\gamma_1=1$ and $\gamma_{2}=x$. In this case by \eqref{MarsId} and  \eqref{hbernsteinbasis}, the Marsden identity becomes
		\begin{equation}\label{eq:marsden}
			\prod_{i=0}^{n-1} \frac{(x - t + ih)}{(b - a + ih)}
			=
			\sum_{j=0}^{n} (-1)^{j}
		 \,
			\frac{B_{n-j}^{\,n}\!\big(x;\,[a-(n-1)h,\, b]; -h\big)}{\binom{n}{j}}\;
			B_{j}^{\,n}\!\big(t;\,[a,b];h\big),
		\end{equation}
		where 	$B_{k}^{n}\!\big(t;\,[a,b];h\big)$ are the $h$-Bernstein basis functions in \cite{ron1}.
		Also, for  $h =0$, from \eqref{MarsId} and \eqref{gammaber}, we get
		\begin{equation}\label{eq:marsden2}
			\left( \frac{d(t, x)}{d(a, b)} \right)^{n}
			=
			\sum_{k=0}^{n} (-1)^{k} 
		\,
			\frac{B_{n-k}^{\,n}(x,[a,b])\,
			B_{k}^{\,n}(t, [a,b])}{\binom{n}{n-k}},
		\end{equation}
		where $B_{k}^{\,n}(x, [a,b])$ are the  Bernstein basis functions in  \eqref{gammaber}.
		\subsection{Interpolation}
	\begin{corollary}\label{cor:interpolation}
		Let $G\in \pi_{n}(\gamma_{1},\gamma_{2}) $ with control points $P_0,\ldots,P_n$
		over the interval $[a,b]$. Then $G$ interpolates its first and last control
		points. $G(a)=P_0$ and $G(b)=P_n$.
	\end{corollary}
	
	\begin{proof}
		By the dual functional property,  the control points
		$P_0,\dots,P_n$ of $G$ are given by
		\[
		P_k = g(\Gamma(a-kh),\dots,\Gamma(a-(n-1)h),
		\Gamma(b),\Gamma(b-h),\dots,\Gamma(b-(k-1)h);h),
		\qquad k=0,\dots,n,
		\]
		where $g$ is the  $h$--$\gamma$ blossom of $G$. Evaluating the curve at $x=a$ and $x=b$ yields
		\[
		G(a)=g(\Gamma(a),\dots,\Gamma(a-(n-1)h);h) = P_0,\qquad
		G(b)=g(
		\Gamma(b),\Gamma(b-h),\dots,\Gamma(b-(n-1)h);h) = P_n.
		\]
	\end{proof}
	
	\begin{corollary}[Interpolation]
		Let $G \in \pi_{n}(\gamma_{1},\gamma_{2})$  with
		control points $P_0,\ldots,P_n$ over the interval $[a,b]$. If $b=a-nh$ and $h\neq 0$, then $G$ interpolates all its
		control points. In particular,
	$G(a-kh)=P_k$ for  $k=0,\ldots,n$.
	\end{corollary}
	
	\begin{proof}
		Let $g$ be the  $h$--$\gamma$ blossom of $G$. Then, by the dual functional property and the $h$--$\gamma$ diagonal evaluation
		\begin{align*}
			P_k
			&= g(\Gamma(a-kh),\ldots,\Gamma(a-(n-1)h),
			\Gamma(a-nh),\ldots,\Gamma(a-(n+k-1)h))=G(a-kh).
		\end{align*}
	\end{proof}
	\subsection{Degree Elevation}
	We now derive degree elevation formulas for  $h$--$\gamma$ Bernstein bases. Typically, the spaces $\pi_n(\gamma_1, \gamma_2)$ are not nested. But if $1 \in \pi_k(\gamma_1, \gamma_2)$, then $\pi_n(\gamma_1, \gamma_2) \subset \pi_{n+k}(\gamma_1, \gamma_2)$. For the polynomial spaces in Example \ref{ex1}, $1 \in \pi_1(\gamma_1, \gamma_2)$ so there are degree elevation formulas from $\pi_n(\gamma_1, \gamma_2)$ to $\pi_{n+1}(\gamma_1, \gamma_2)$. For the trigonometric and hyperbolic spaces in Examples \ref{e2}--\ref{ex4}, $1 \in \pi_2(\gamma_1, \gamma_2)$ so there are degree elevation formulas from $\pi_n(\gamma_1, \gamma_2)$ to $\pi_{n+2}(\gamma_1, \gamma_2)$. Next, we consider these two special cases.
	\begin{proposition}\label{thm:degpol}
		Let $\gamma_{1}(x)=1$ and $\gamma_{2}(x)=x$. Then for the space of polynomials of degree $n$, the $h$-Bernstein basis functions satisfy
		the degree elevation formula
		\begin{align}\label{def1}
				B_k^n(x,[a,b];h)
			=
			\frac{n+1-k}{n+1}\,B_k^{n+1}(x,[a,b];h)
			+
			\frac{k+1}{n+1}\,B_{k+1}^{n+1}(x,[a,b];h).
		\end{align}
	\end{proposition}
	
	\begin{proof}
		From  formula (\ref{hbernsteinbasis}) for
		$B_k^n(x,[a,b];h)$,
		\[
		\frac{b-x+(n-k)h}{\,b-a+nh\,}\,
		B_k^n(x,[a,b];h)
		=
		\frac{\binom{n}{k}}{\binom{n+1}{k}}\,
		B_k^{n+1}(x,[a,b];h)
		=
		\frac{n+1-k}{n+1}\,
		B_k^{n+1}(x,[a,b];h),
		\]
		
		\[
		\frac{x-a+kh}{\,b-a+nh\,}\,
		B_k^n(x,[a,b];h)
		=
		\frac{\binom{n}{k}}{\binom{n+1}{k+1}}\,
		B_{k+1}^{n+1}(x,[a,b];h)
		=
		\frac{k+1}{n+1}\,
		B_{k+1}^{n+1}(x,[a,b];h).
		\]
		Adding these two formulas yields the result.
	\end{proof}
	Notice that the coefficients $\frac{n+1-k}{n+1}$ and $\frac{k+1}{n+1}$ in \eqref{def1}
	are independent of $a,b,h$. Proposition \ref{thm:degpol} is a special case of Proposition 3.1 in \cite{ron3}.
	\begin{proposition} Let $\gamma_1(x)=\cos x$, $\gamma_{2}(x)=\sin x$ and $h=0$.
		The trigonometric Bernstein basis functions satisfy the	degree elevation formula
		\begin{align}\nonumber
			B_k^n(x,[a,b];\gamma)
			=&
			\frac{(n+2-k)(n+1-k)}{(n+1)(n+2)}\,B_k^{n+2}(x,[a,b];\gamma)
			+
			\frac{2\cos(b-a)(k+1)(n+1-k)}{(n+1)(n+2)}\,B_{k+1}^{\,n+2}(x,[a,b];\gamma)\\& \label{def2}
		+
			\frac{(k+1)(k+2)}{(n+1)(n+2)}\,B_{k+2}^{\,n+2}(x,[a,b];\gamma).
		\end{align}
	\end{proposition}
	
	\begin{proof}
		For $h=0$, the trigonometric Bernstein basis functions are given by (\ref{gammaber})
		\begin{align}\label{trigbasis}
				B_k^n(x,[a,b];\gamma)
			= \binom{n}{k}
			\frac{\sin^k(x-a)\,\sin^{\,n-k}(b-x)}{\sin^n(b-a)} .
		\end{align}
		Consider the trigonometric identities
		\begin{align*}
			\frac{\sin(x-a)+\sin(b-x)}{2\sin\big(\tfrac{b-a}{2}\big)}
			= \cos\!\Big(x-\frac{a+b}{2}\Big), \qquad
			\frac{\sin(x-a)-\sin(b-x)}{2\cos\big(\tfrac{b-a}{2}\big)}
			= \sin\!\Big(x-\frac{a+b}{2}\Big).
		\end{align*}
		Squaring and adding yields
		\[
		\left(
		\frac{\sin(x-a)+\sin(b-x)}{2\sin(\tfrac{b-a}{2})}
		\right)^2
		+
		\left(
		\frac{\sin(x-a)-\sin(b-x)}{2\cos(\tfrac{b-a}{2})}
		\right)^2
		=1,
		\]
		which  is equivalent to
		\[
	\frac{\sin^2(x-a)}{\sin^2(b-a)}
	+
	\frac{2\cos(b-a)\,\sin(x-a) \sin(b-x)}{\sin^2(b-a)}
	+
	\frac{\sin^2(b-x)}{\sin^2(b-a)} = 1.
		\]
Multiplying both sides of this equation by $B_k^n(x,[a,b];\gamma)$ and applying \eqref{trigbasis} yields formula  \eqref{def2}.
	\end{proof}

  \subsection{Subdivision}
  \begin{proposition}[Subdivision Algorithm]\label{subdivision}
  	Let $P_k$, $k=0,\dots,n$, be the control points for an $h$-- $\gamma$ B\'ezier curve
  	$G(x) \in \pi_n(\gamma_1,\gamma_2)$ over the interval $[a,b]$.
  	Then the control points for the restriction of $G(x)$ to the subintervals $[a,t]$ and $[t,b]$
  	are given by
  	\[
  	L_k=P^k_0(t)\quad\text{and}\quad R_k=P_k^{\,n-k}(t),\qquad k=0,\dots,n,
  	\]
  	where $P^k_0(t)$  are the intermediate points of the evaluation algorithm for $G(x)$ with $\sigma(i)=i$ and $P_k^{\,n-k}(t)$ are the intermediate points of the evaluation algorithm for $G(x)$ with $\sigma(i)=n+1-i$,   evaluated at $x = t$.
  \end{proposition}
  \begin{proof}
  	Let $g$ be the  $h$--$\gamma$  blossom of $G$. By the dual functional property (Theorem \ref{dualfuncprop}) and \eqref{Pikx-sig}, the control points of the left segment over $[a,t]$ are
  	\[
  	L_k =g(\Gamma(a-kh), \ldots,\Gamma( a-(n-1) h), \Gamma(t), \Gamma(t-h), \ldots, \Gamma(t-(k-1) h) ; h)=P^k_0(t).
  	\]
  	Similarly, for the right segment over $[t,b]$ the control points are
  	\[
  	R_k =g(\Gamma(t-kh), \ldots,\Gamma( t-(n-1) h), \Gamma(b), \Gamma(b-h), \ldots, \Gamma(b-(k-1) h) ; h)=P_k^{n-k}(t).\]
  \end{proof}
  
  \textit{Recursive Midpoint Subdivision Algorithm:}
We can now construct a recursive subdivision algorithm for  $h$--$\gamma$ B\'{e}zier curves 
$G$ defined over the interval $[a,b]$, as follows.
First take $t=\tfrac{a+b}{2}$, the midpoint of $a$ and $b$, then subdivide $G$ into two curve segments. Iteratively subdivide each curve segment at the midpoint of its parametric domain. Then after each iteration the number of new curve segments doubles and at the $N$th iteration we will have $2^{N}$ curve segments 
corresponding to the subintervals $[t_{i},t_{i+1}]$, $i=0,1,\dots,2^{N}-1$, where 
$t_{i}=a+\tfrac{i}{2^{N}}(b-a)$. The control polygons for all these curve segments form a piecewise  linear approximation to the original $h$--$\gamma$ B\'{e}zier curve $G$. 


%

 \begin{theorem}\label{theorem_subd}
 	The control polygons generated by the recursive midpoint subdivision algorithm converge 
 	pointwise to the original  $h$--$\gamma$ B\'{e}zier  curve.
 \end{theorem}
 \begin{proof}
 	Let $G \in \pi_{n}(\gamma_{1},\gamma_{2})$ be an $h$--$\gamma$ B\'{e}zier curve over the interval $[a,b]$ 
 	and let $g$ denote the $h$--$\gamma$  blossom of $G$. At the $N$-th iteration of the recursive subdivision algorithm, 
 	choose and fix any $c \in \{t_{i}\}_{i=0}^{2^{N}-1}$ and consider the control points of the curve segment corresponding to the subinterval $\big[c,\,d_N\big],\; d_N=c+\frac{b-a}{2^{N}}$. By the dual functional property (Theorem \ref{dualfuncprop}), these control points are given by
 	\[
 	g\left(\Gamma(c-k h), \ldots, \Gamma(c-(n-1) h), \Gamma\left(d_N\right), \ldots, \Gamma\left(d_N-(k-1) h\right) ; h\right),
 	\quad k=0,\dots,n.
 	\]
 	Since $\gamma_{1}$ and $\gamma_{2}$ are continuous, taking the limit as $N \to \infty$ and using the continuity of the $h$--$\gamma$ blossom we obtain
 \begin{align*}
 		&\lim_{N \to \infty}
 	g\left(\Gamma(c-k h), \ldots, \Gamma(c-(n-1) h), \Gamma\left(d_N\right), \ldots, \Gamma\left(d_N-(k-1) h\right) ; h\right)\\&
 	= 	g\left(\Gamma(c-k h), \ldots, \Gamma(c-(n-1) h), \Gamma\left(c\right), \ldots, \Gamma\left(c-(k-1) h\right) ; h\right)\\&
 	= G(c).
 \end{align*}
 	Thus, the control polygons converge pointwise to $G$ at every dyadic subdivision point $c$ (i.e., points of the form $j/2^r, \; j \in \mathbb{Z},\ r \in \mathbb{N}$.). Since the dyadic points are dense in $[a,b]$, for any  $x \in [a,b]$ there exists a sequence of dyadic points approaching $x$. The continuity 
 	of the function $G$ and its $h$--$\gamma$ blossom $g$ imply that the convergence at dyadic points extends to all $x \in [a,b]$, so the control polygons generated by recursive subdivision converge pointwise to the $h$--$\gamma$ B\'{e}zier curve $G$.
 \end{proof}
  
  \begin{theorem}\label{theorem_subd2} Let $\gamma_{1},\gamma_{2} \in C^1[a,b]$. Then
  the control polygons generated 
  	by recursive midpoint subdivision converge uniformly to the original 
  	 $h$--$\gamma$ B\'{e}zier curve.
  \end{theorem}
  
  \begin{proof}
  	Consider an $h$--$\gamma$ B\'{e}zier curve $G \in \pi_{n}(\gamma_{1},\gamma_{2})$ defined over the interval $[a,b]$ and let $g$ be its $h$--$\gamma$ blossom. 
  	If we subdivide $G$ at a point $x \in [a,b]$, then by the dual functional property (Theorem \ref{dualfuncprop}) the control points of the left segment are
  	\[
  	L_{k} =g\left(\Gamma(a-kh), \ldots,\Gamma( a-(n-1) h), \Gamma(x), \Gamma(x-h), \ldots, \Gamma(x-(k-1) h) ; h\right), 
  	\quad k=0,\dots,n.
  	\]
  	By the multilinear property of the $h$--$\gamma$ blossom,
 \begin{align*}
 		&L_{k+1}-L_{k} \\&= g\left(\Gamma(a-(k+1)h), \ldots,\Gamma( a-(n-1) h), \Gamma(x), \Gamma(x-h), \ldots, \Gamma(x-(k-1) h), \Gamma(x-kh)-\Gamma(a-kh); h\right).
 \end{align*}
 By the mean value theorem, there exists $\xi_1,\, \xi_2 \in (a,x)$ such that \[
\Gamma(x-kh)-\Gamma(a-kh) 
= (x-a)\,(\gamma'_{1}(\xi_1-kh),\gamma'_{2}(\xi_2-kh)).
\]
 Let 
  \begin{align*}
  	M = \max_{0 \le k \le n-1} \;
  	\Bigl|
  	g\bigl(
  	&\Gamma(y-(k+1)h), \ldots, \Gamma(y-(n-1)h), \\
  	&\Gamma(x), \Gamma(x-h), \ldots, \Gamma(x-(k-1)h), \\
  	&(\gamma'_1(z_1-kh), \gamma'_2(z_2-kh)); h
  	\bigr)
  	\Bigr|
  \end{align*}
  	over all $x,y,z_1,z_2 \in [a,b]$. It follows that $
  	|L_{k+1}-L_{k}| \leq  M |x-a| \leq  M |b-a|,
  	\quad k=0,\dots,n-1$. Hence
  	\begin{align}\label{subequa}
  			\sum_{k=0}^{n-1} |L_{k+1}-L_{k}| < n|b-a|M.
  	\end{align}
  	A similar argument shows that for the right segment with control points $R_{k}$,
  	\[
  	\sum_{k=0}^{n-1} |R_{k+1}-R_{k}| < n |b-a| M.
  	\]
  	Now let $\widetilde{G}$ denote a segment of the original curve $G$ obtained after $m$ iterations of midpoint subdivision, and let $L(t)$ denote its associated control polygon. Then $\widetilde{G}$ is the
  	restriction of $G$ over a subinterval $[\tilde{a},\tilde{b}] \subset [a,b]$ of length $(b-a)/2^{m}$, and by Corollary \ref{cor:interpolation}, $G$ and $L$ coincide at $\tilde{a}$ and $\tilde{b}$. Therefore, for any $t \in [\tilde{a},\tilde{b}]$, 
  \begin{align*}
  	|G(t)-L(t)|
  	\leq |G(t)-L(\tilde a)| + |L(\tilde a)-L(t)|=|G(t)-\widetilde{G}(\tilde a)| + |L(\tilde a)-L(t)|.
  \end{align*}
  	By the mean value theorem,
  	\[
  	|G(t)-G(\tilde{a})| \leq |t-\tilde{a}|\max_{\tau \in [\tilde{a},\tilde{b}]}|G'(\tau)|.
  	\]
  	On the other hand, by \eqref{subequa}
  	\[
  	|L(\tilde{a})-L(t)| \leq n|\tilde{b}-\tilde{a}|M.
  	\]
  	Combining to the last two estimates yields
  	\[
  	|G(t)-L(t)| \leq \frac{(b-a)\tilde{M}}{2^{m}},
  	\]
  	where $\tilde{M}=\max_{\tau \in [a,b]}|G'(\tau)|+nM$.  
  	
  	Since $\tilde{M}$ is independent of  $m$, this argument shows that the control polygons generated by recursive midpoint subdivision converge uniformly to the  $h$--$\gamma$ B\'{e}zier curve $G$ at an exponential rate.
  \end{proof}

\appendix
\section*{Appendix A: Conditions Guaranteeing the Linear Independence of the Functions $\boldsymbol{G_{n,k}(t;h)}$} \label{appxA}
\addcontentsline{toc}{section}{Appendix A: Conditions Guaranteeing the Linear Independence of the Functions $G_{n,k}$}
\setcounter{equation}{0}
\renewcommand{\theequation}{A.\arabic{equation}}

\setcounter{theorem}{0}
\renewcommand{\thetheorem}{A}

We will use the notation $\rho^{[0]}(t)=t$, $\rho^{[1]}(t)=\rho(t)=t-h$, and $\rho^{[j]}(t)=\rho(\rho^{[j-1]}(t))$, $j>1$; 
the column vector notation $\vec{v}$ and $\vec v^T$ for the transposed of $\vec{v}$; the size of a finite set $J$ 
will be denoted by $|J|$, and the sum of the elements of $J$ by $s(J)$; the determinant of a square matrix $A$ will be denoted by $|A|$;  
and the $q$-binomial coefficients (polynomials in $q$) notation $\gauss{n}{k}=(q;q)_n/((q;q)_k(q;q)_{n-k})$, $k=0,\ldots,n$, where 
$(a;q)_n=\prod_{j=0}^n(1-q^ja)$, $n\in\mathbb{N}$, $(a;q)_0=1$. In what follows we omit the dependence on the variable $t$ since it is not 
needed. Then the functions in \eqref{hbasis} are
\begin{align}
	\label{Gnk-def} 
	G_{n,k}=\sum_{J\subseteq\{0,\ldots,n-1\},\,|J|=k}\,\,\prod_{j\in J}\gamma_1(\rho^{[j]})
	\prod_{j\in J'=\{0,\ldots,n-1\}\setminus J}\gamma_2(\rho^{[j]}), \qquad k=0,\ldots,n.  
\end{align}
\begin{proposition} $($Linear Independence$)$ 
	Let $C=C(\rho)$, and let $\lambda_1$ and $\lambda_2$ be the eigenvalues of $C$. Suppose that $|C| \neq 0$ and that if 
	$C$ is diagonalizable then $q=\lambda_1/\lambda_2$ is not a zero of any of the $q$-polynomials $\gauss{n}{k}$, $k=1,\ldots,n-1$.  
	Then the functions $G_{n,k}$, $k=0,\ldots,n$ are linearly independent. 
\end{proposition} 
\begin{proof} 
	With $\vec{\gamma}=(\gamma_1,\gamma_2)^T$ we have $\vec{\gamma}(\rho)=C\vec{\gamma}$, hence $\vec{\gamma}(\rho^{[j]})=C^j\vec{\gamma}$, $j\in\mathbb{N}_0$. 
	We consider two cases. \\ 
	
	\noindent\underline{Case 1}. $C$ is diagonalizable. Then $C=M^{-1}DM$, where $D$ is a diagonal matrix with diagonal entries 
	$\lambda_1$ and $\lambda_2$. Set $\vec\vp=(\vp_1,\vp_2)^T=M\vec\gamma$. Then $\vec\vp(\rho)=D\vec\vp$ and 
	$\vec\vp(\rho^{[j]})=D^j\vec\vp$, $j\in\mathbb{N}_0$. Consider the functions 
	\begin{eqnarray}
		\label{Hnk-def} 
		\begin{aligned}
			H_{n,k}&=\sum_{J\subseteq\{0,\ldots,n-1\},\,|J|=k}\,\,\prod_{j\in J}\vp_1(\rho^{[j]})\prod_{j\in J'}\vp_2(\rho^{[j]}) \\ 
			&=\vp_1^k\vp_2^{n-k}\lambda_2^{{n\choose 2}}\sum_{J\subseteq\{0,\ldots,n-1\},\,|J|=k}q^{s(J)}=\vp_1^k\vp_2^{n-k}\lambda_2^{n\choose 2}q^{k\choose 2}\gauss{n}{k},  
		\end{aligned}
	\end{eqnarray}
	$k=0,\ldots,n$, where we used the $q$-binomial formula $(-x;q)_n=\prod_{j=0}^{n-1}(1+q^jx)=\sum_{k=0}^n\gauss{n}{k}x^k$. 
	By assumption $\gauss{n}{k}\neq 0$ for any $k=1,\ldots,n-1$, which is so if $q\notin\{e^{2\pi i \nu/n}\}_{\nu=1}^{n-1}$.  
	Then \eqref{Hnk-def} shows that $\pi_n(\vp_1,\vp_2)={\rm span}\{H_{n,k}\}_{k=0}^n$. 
	Since $\vec\gamma=M^{-1}\vp$, we also have $\pi_n(\gamma_1,\gamma_2)=\pi_n(\vp_1,\vp_2)$. 
	
	Now we will show that $\{H_{n,k}\}_{k=0}^n\subset{\rm span}\{G_{n,k}\}_{k=0}^n$. This is obvious if $C=D$. So suppose that 
	$C\neq D$. Let $M=[m_{j,l}]$ and set $\al_{\nu}=m_{\nu,1}/m_{\nu,2}$, $\nu=1,2$, and $r=\gamma_1/\gamma_2$. 
	Using \eqref{Hnk-def} and \eqref{Gnk-def}, we derive  
	\begin{eqnarray}
		\label{Hnk-C1}
		\begin{aligned}
			H_{n,k}&=\sum_{J\subseteq\{0,\ldots,n-1\},\,|J|=k}\,\,\prod_{j\in J}\bigl(m_{1,1}\gamma_1(\rho^{[j]})+m_{1,2}\gamma_2(\rho^{[j]})\bigr)
			\prod_{j\in J'}\bigl(m_{2,1}\gamma_1(\rho^{[j]})+m_{2,2}\gamma_2(\rho^{[j]})\bigr) \\ 
			&=m_{1,2}^k m_{2,2}^{n-k}\Bigl\{\prod_{j=0}^{n-1}\gamma_2(\rho^{[j]})\Bigr\}\sum_{J\subseteq\{0,\ldots,n-1\},\,|J|=k}\,\,
			\prod_{j\in J}(\al_1 r(\rho^{[j]})+1)\prod_{j\in J'}(\al_2 r(\rho^{[j]})+1) \\ 
			&=m_{1,2}^k m_{2,2}^{n-k}\Bigl\{\prod_{j=0}^{n-1}\gamma_2(\rho^{[j]})\Bigr\}a_{n,k,l}(M)s_{n,l}(r,r(\rho),\ldots,r(\rho^{[n-1]}))
			=m_{1,2}^k m_{2,2}^{n-k}\sum_{l=0}^n a_{n,k,l}(M)G_{n,l}, 
		\end{aligned} 
	\end{eqnarray} 
	where $s_{n,l}$ are the elementary symmetric multilinear  functions of $n$ variables and 
	\begin{align}
		\nonumber 
		a_{n,k,l}(M)=\sum_{\nu=0}^k \al_1^{\nu}\al_2^{l-\nu}{l\choose{\nu}}{{n-l}\choose{k-\nu}}, \qquad l=0,\ldots,n. \end{align} 
	Therefore $\pi_n(\gamma_1,\gamma_2)\subset{\rm span}\{G_{n,k}\}_{k=0}^n$. Since ${\rm dim}(\pi_n(\gamma_1,\gamma_2))=n+1$, it follows that  
	$\{G_{n,k}\}_{k=0}^n$ are linearly independent. \\ 
	
	\noindent\underline{Case 2.} $C$ is not diagonalizable. Then $C=M^{-1}UM$, where 
	$U=\left[\begin{array}{cc} \la & 1 \\ 0 & \la \end{array}\right]$. Again set  $\vec\vp=M\vec\gamma$. Then 
	$\vec\vp(\rho^{[j]})=U^j\vec\vp$, where 
	$U^j=\la^{j-1}\left[\begin{array}{cc} \la & j \\ 0 & \la \end{array}\right]$, $j\in\mathbb{N}_0$. In this case the functions in 
	\eqref{Hnk-def} become 
	\begin{eqnarray}
		\label{Hnk-C2}
		\begin{aligned}
			H_{n,k}&=\sum_{J\subseteq\{0,\ldots,n-1\},\,|J|=k}\,\,\prod_{j\in J}\bigr(\la^j\vp_1+\la^{j-1}j\vp_2\bigr)
			\prod_{j\in J'}(\la^j\vp_2) \\ 
			&=\la^{n\choose 2}\vp_2^n\sum_{J\subseteq\{0,\ldots,n-1\},\,|J|=k}\,\,\prod_{j\in J}(\vp_1/\vp_2+j/\la) \\ 
			&=\la^{n\choose 2}\vp_2^n\sum_{l=0}^k{{n-k+l}\choose l}(\vp_1/\vp_2)^l \la^{l-k}s_{n,k-l}(0,\ldots,n-1). 
		\end{aligned}
	\end{eqnarray} 
	With $\vec H_n=(H_{n,0},\ldots,H_{n,n})^T$ and $\vec\vp_n=(\vp_2^n,\vp_1\vp_2^{n-1},\ldots,\vp_1^n)^T$ it follows from  
	\eqref{Hnk-C2} that $\vec H_n=A_n\vec\vp_n$, where the $(n+1)\times(n+1)$ matrix $A_n=A_n(\la)$ is lower-triangular 
	and has diagonal entries $\la^{n\choose 2}{n\choose k}$, $k=0,\ldots,n$. Therefore $A_n$ is invertible and 
	$\pi_n(\vp_1,\vp_2)={\rm span}\{H_{n,k}\}_{k=0}^n$. The rest of the proof for Case 2 is the same as for Case 1. 
\end{proof}

\section*{Acknowledgments}

This work was carried out while the first author was visiting Rice University and was supported by the Scientific and Technological Research Council of T\"{u}rkiye (T\"{U}B\.{I}TAK) under the B\.{I}DEB-2219 International Postdoctoral Research Fellowship Program.


\begin{thebibliography}{00}

	
	\bibitem{cetin2}
	Di{\c{s}}ib{\"u}y{\"u}k \c{C}., Goldman R. (2015). A unifying structure for polar
	forms and for Bernstein B{\'e}zier curves, J.  Approx. Theory,
	192,  234--249.
		\bibitem{cetin3}
	Di{\c{s}}ib{\"u}y{\"u}k \c{C}., Goldman R. (2016). A unified approach to non-polynomial B-spline curves based on a novel variant of the polar form. Calcolo, 53(4), 751-781.
	
	
	\bibitem{lyche}
	Lyche, T. (1999). Trigonometric splines; a survey with new results. Shaping Preserving Representations in Computer-Aided Geometric Design, 201-227.
	
	\bibitem{goldman1}
	Goldman, R. (1985). P\'olya's urn model and computer aided geometric design. SIAM J. Alg. Disc. Meth. 6, 1–28.
	
	\bibitem{goldman2}
	Goldman, R., Barry, P. (1991). Shape parameter deletion for P\'olya curves. Numer. Algorithms 1, 121–137.
	
	
	
	
	\bibitem{ron3}Goldman, R., Simeonov, P. (2015). Two essential properties of $(q,h)$-Bernstein-B\'ezier curves. Applied Numerical Mathematics, 96, 82-93.
	
		\bibitem{gonsor}
	Gonsor, D., Neamtu, M. (1994). Non-polynomial polar forms. Curves and Surfaces in Geometric Design (Chamonix-Mont-Blanc, 1993), 193-200.

	
	
	\bibitem{ron1}
	Simeonov, P., Zafiris, V.,  Goldman, R. (2011). $h$-Blossoming: A new approach to algorithms and identities for $h$-Bernstein bases and $h$-B\'{e}zier curves. Computer Aided Geometric Design, 28(9), 549-565.
	
	\bibitem{stancu1}
	Stancu, D. (1968). Approximation of functions by a new class of linear polynomial operators. Rev. Roumaine Math. Pures Appl. 13, 1173–1194.
	
	\bibitem{stancu2}
	Stancu, D. (1984). Generalized Bernstein approximating operators. In: Itinerant Seminar on Functional Equations, Approximation and Convexity. Cluj-Napoca,
	1984. Univ. "Babes-Bolyai", Cluj-Napoca, pp. 185–192. Preprint 84-6.
	
	


	
\end{thebibliography}
\end{document}